\documentclass[a4paper,12pt]{amsart}
\usepackage[utf8]{inputenc}
\usepackage[T1]{fontenc}
\usepackage[english]{babel}
\usepackage{amssymb,amsthm,times}
\usepackage{xcolor}
\usepackage{hyperref}
\usepackage{cancel}
\usepackage{graphicx}

\usepackage{float}
\usepackage{epstopdf}

\newcommand{\NormOnly}{\| \cdot \|}

\newcommand{\co}{\operatorname{co}}

\newcommand{\weak}{\hbox{weak}}
\newcommand{\ws}{weak\mbox{$^*$}}

\newcommand{\R}{\mathbb R}

\newcommand{\N}{\mathbb{N}}

\newtheorem{thm}{Theorem}
\newtheorem{lem}[thm]{Lemma}

\newtheorem{rem}[thm]{Remark}

\newtheorem{que}[thm]{Question}
\date{}

\title{\bf   Conic James' compactness Theorem}

\author{J. Orihuela}
\address{Departamento de Matem\'{a}ticas, Universidad de Murcia,
30100 Espinardo (Murcia), Spain} \email{joseori@um.es}

\thanks{Partially supported by Ministerio de Economía y Competitividad and FEDER project MTM2014-57838-C2-1-P; and Fundación Séneca  CARM, project 19368/PI/14.
}



\begin{document}

\maketitle

\begin{abstract}
Our main result, which answers for arbitrary Bananch spaces a question posed in \cite{cas-ori-per} is the following:

{\it Let $E$ be a Banach space  and  $D$ be a weakly  compact subset of $E$ with $0\notin D$. If $A$ is a bounded subset of $E$ such that every $x^*\in E^*$ with  $x^*(D) >0$ attains its supremum on $A$, then $A$ is weakly relatively compact.}

\end{abstract}

\label{intro}
\section{Introduction}
The well-known James' theorem \cite{jam} claims that {\it a bounded closed convex set $C$ in a Banach space $E$ is weakly compact if and only if every $x^{\ast} \in E^{\ast}$ attains its supremum on $C$}. 

The aim of this paper is to complete the answer given in \cite{cas-ori-per} to the following question raised by Delbaen: 
\begin{que}
\label{que:Delbaen}
Let $E$ be a Banach space and $A$ be a bounded, convex and closed subset  of $E$ with $0\notin A$. Let us assume that for every 
$x^*\in E^*$ with $$\inf x^*(A)>0 ,$$
the infimum of $x^*$ on A is attained. Is the set  $A$  weakly compact?
\end{que}

For Banach spaces with a $\ws$ convex-block compact dual unit ball we gave a positive answer in \cite{cas-ori-per} called 
One-Sided James' Theorem. Thus in particular for Banach spaces with $\ws$ sequentially compact dual unit ball. We do not know if the result is true for arbitrary Banach spaces.
Our main theorem here gives a  positive answer if the hypothesis $\inf x^*(A)>0$ is replaced by requiring that
$x^*(D) >0$ for some fixed relatively weakly compact set of directions $D$, it reads as follows:
\begin{thm}
\label{MainConic}
Let $E$ be a Banach space  and  $D$ be a weakly  compact subset of $E$ with $0\notin D$. If $A$ is a bounded subset of $E$ such that every $x^*\in E^*$  with $x^*(D) >0$ attains its supremum on $A$, then $A$ is weakly relatively compact.
\end{thm}

This result answers a question posed in \cite{cas-ori-per} where the same result is proved for Banach spaces with a $\ws$ convex-block compact dual unit ball. Let us remark that Delbaen's problem was motivated by some questions on risk measures in the framework of financial mathematics. Applications
and extensions of James' theorem in this field can be found for instance  in \cite{ori-rui1} and \cite{ori-rui2}. A general theorem for level sets of functions is in \cite{ori-rui1} and \cite{sai}. We very much hope other applications will come soon. An uptodate account on James' compactness theorem can be found in \cite{cas-ori-rui}

\subsection{Notation and terminology}

Most of our notation
and terminology are standard and can be
found in our standard references for Banach
spaces~\cite{fab-haj-mont-ziz}.  

Unless otherwise stated, $E$ will denote a Banach space with the norm $\NormOnly$. Given a subset~$S$
of a vector space, we write $\co{(S)}$,  to denote, respectively, its convex hull. If $(E,\NormOnly)$ is a normed space then $E^*$ denotes its topological dual. If $S$ is a subset of $E^*$, then $\sigma(E,S)$ denotes the topology of pointwise
convergence on $S$. Dually, if $S$ is a subset of $E$, then
$\sigma(E^*,S)$ is the topology for $E^*$ of pointwise convergence
on $S$. In particular $\sigma(E,E^*)$ and $\sigma(E^*,E)$ are the
weak ($\omega$)\index{weak topology}\index{topology!weak} and weak$^*$ ($\omega^{\ast}$)\index{weak$^*$ topology} topologies respectively.  

Given $x^{*}\in E^{*}$ and $x\in
E$, we write $\langle x^{*}, x \rangle = \langle x, x^{*} \rangle = x^{*}(x)$ for the
evaluation of~$x^{*}$ at~$x$. If $x\in E$ and $\delta >0$ we denote by $B(x,\delta)$   (or $B[x,\delta]$) the open (resp. closed) ball centered at $x$ of radius $\delta$. To simplify, we will simply write $B_{E}:=B[0,1]$; and the unit sphere \mbox{$\{x\in E\colon \|x\|=1\}$} will be denoted by $S_E$. An element $x^{\ast} \in E^{\ast}$ is \emph{norm-attaining} if there is $x \in B_{E}$ with $x^{\ast}(x) = \| x^{\ast}\|$. The set of norm-attaining functionals of $E$ is normally denoted by $NA(E)$.

If $(x_n)$ is a  bounded sequence in the Banach space $E$ we denote by
$$\hbox{co}_\sigma= \{\sum_{n=1}^\infty \xi_n x_n: \xi_n\geq 0 \hbox{ and } \sum_{n=1}^\infty \xi_n =1\} $$
the $\sigma$-convex hull of the sequence $(x_n)$

\section{The separable case}
In this section we present a proof of our main result for the separable case. Of course it follows form our results in \cite{cas-ori-per}.
Nevertheless it should be of interest for deeply understanding of matters.
We are going to apply the following result, see Theorem 3 in \cite{ori-rui2} or Corollary 10.6 in \cite{cas-ori-rui}

\begin{thm}[\textbf{Inf-liminf theorem in } $\mathbf{\mathbb{R}^X}$]\label{slst}
Assume that $X$ is a nonempty set, $\{f_n\}_{n\ge
1}$ is a pointwise bounded sequence in $\mathbb{R}^X$ and $Y$ is a
subset of $X$ with the property
$$
\hbox{for every } g \in \mathrm{co}_{\sigma_p}\{f_n\colon n \ge 1\}
\hbox{ there exists } y\in Y \hbox{ with } g (y)=\inf_X(g).
$$
where
$$
{\rm co}_{\sigma_p}\{f_n\colon\ n \ge 1\}:=\left\{ \sum_{n=1}^\infty \lambda_n f_n : \hbox{ for all }n\ge 1, \ \lambda_n\ge 0  \hbox{ and } \sum_{n=1}^\infty \lambda_n=1 \right\},
$$
and the functions $\sum_{n = 1}^\infty\lambda_n f_n \in \mathbb{R}^X$ above are
pointwise defined on $X$, i.e. for every $x\in X$ the absolutely convergent series $$\sum_{n = 1}^\infty\lambda_n f_n(x)$$
defines the function $\sum_{n = 1}^\infty\lambda_n f_n: X\rightarrow \R$.

Then
$$
\inf_X\left(\liminf_n f_n\right)=\inf_Y\left(\liminf_n f_n\right).
$$
\end{thm}

\begin{lem}\label{previo}
 Lat $E$ be a separable Banach space and $A\subset E$ a closed and convex subset. Let $x_0^{**}$ in $E^{**}\setminus E$ be a $\weak^*$-cluster point of $A$ and $D\subset E$ such that $0\notin \overline{{\rm co}(A\cup D)}$. Then  there is a $\weak^*$-convergente sequence $\{x_n^*\}_{n \ge 1}$ in $B_{E^*}$ and $0< \alpha < \beta$ such that
\begin{equation}\label{c1}
|\langle x^*_n, x^{**}_0\rangle| < \alpha
\end{equation}
whenever  $n\ge 1$, and
\begin{equation}\label{c2}
\lim_n \langle x^*_n ,x \rangle > \beta > \alpha
\end{equation}
for any  $x\in \overline{{\rm co}(A\cup D)}$.
\end{lem}

\begin{proof}
The compact segment line $[-x^{**}, +x^{**}]$ can be strictly separated from the closed convex set $C:=\overline{{\rm co}(A\cup D)}$ in $E^{**}$ by Hahn-Banach Theorem. So there exists a continuous linear functional $x^{***}\in B_{E^{***}}$ satisfying  $$x^{***}(C)> \beta > \alpha > x^{***}([-x^{**}_0, + x^{**}_0])$$
for some $0<\alpha<\beta$.
Let us fix a countable dense subset $\{d_n: n\in \N\}$ of $E$ and let us consider,
for every $n\ge 1$, the set
	$$
		V_n:=\big\{y^{***}\in E^{***}: |y^{***}(x^{**}_0)|<\alpha, \ |(y^{***}-x^{***})(d_i)|\leq \frac{1}{n}, \ i=1,2,\dots,n\big\}
	$$
which is a $\weak^*$-open neighborhood of $x^{***}$. Goldstein's theorem permit us to pick up  $x^*_n\in B_{E^*}\cap V_n$ for every $n\geq 1$. We will have the sequence  $\{x^*_n\}_{n\geq 1}$ which clearly satisfies
	$$
		\lim_n\langle x^*_n, d_p\rangle = \langle x^{***}, d_p \rangle
	$$
for every $p\in \N$. The density of $\{d_1,\cdots, d_n,\cdots\}$ in $E$ together the equicontinuity of the sequence
$\{x^*_n\}_{n\geq 1}$ give us the $\weak^*$-convergence of the sequence, thus we see that
    $$
		\lim_n\langle x^*_n, x\rangle = \langle x^{***}, x \rangle
	$$
for all $x\in E$, and
    $$
		|\langle x^*_n, x^{**}_0\rangle |< \alpha
	$$
for every $n\in \N$.
So the sequence $\{x^*_n\}_{n\geq 1}$ satisfies the lemma and the proof is over.

\end{proof}

\begin{thm}\label{basic}
Let $E$ be a separable Banach space and $A$ be closed and convex set. Let us fix a relatively weakly compact set of directions $D\subset B_E\setminus\{0\}$ such that $0\notin \overline{{\rm co}(A\cup D)}$ and for every $x^*\in E^*$  we have that
$$\inf \{x^*(a):a\in A\}$$ is attained whenever $x^*(d)>0 \hbox{ for every } d\in D $.
Then $A$  is $\weak^*$-closed in $E^{**}$ and $A\cap rB_E$ is a weakly compact set for every $r>0$
\end{thm}

\begin{proof}
If $A\cap B_E$ is not weakly relatively compact there is a $\weak^*$-cluster point $x^{**}_0 \in E^{**}\setminus E$ of $A\cap B_E$.
Since $0\notin \overline{{\rm co}((A\cap B_E)\cup D)}$ our Lemma \ref{previo} applies to provide us with a $\weak^*$-convergent sequence $\{x_n^*\}_{n \ge 1}$ to $x^*_0$ in $B_{E^*}$ and numbers $\beta> \alpha >0$ satisfying \eqref{c1} and  \eqref{c2} for $\overline{{\rm co}((A\cap B_E) \cup D)}$. Since the set of directions $D$ is weakly relatively compact and we have
$$x^*_0\in \overline{{\rm co}\{x_n^*: n\in \N\}}^{\tau(E^*,E)}$$
where $\tau(E^*,E)$ is the Mackey topology of the dual pair $\langle E^*, E\rangle$, i.e the topology of uniform convergence on weakly compact convex subsets of $E$, it follows that $x^*_0\in \overline{{\rm co}\{x_n^*: n\geq m\}}^{p_D}$ for all $m\in \N$ where $p_D$ is the seminorm $p_D(x^*):=\sup \{|\langle x^*, d\rangle |: d\in D\}$.
So we can find a sequence of finite subsets of integers $(F_n)$, and convex combinations of scalars 
$\sum_{i \in F_n} \lambda_i^n =1, 0< \lambda_i^n\leq 1$ such that the sequence $y^*_n:=\sum_{i\in F_n}\lambda_i^n x^*_i$ converges uniformly to $x^*_0$ on $D$. Since $x^*_0(d)>\beta$ for all $d\in D$ by \eqref{c2} we can and do assume $y^*_n(d)> \beta > 0$ for all $d\in D$ and $n\in \N$. Thus $y^*$ attains its infimum on $A$ for every $y^* \in {\rm co}_{\sigma_p}\{y^*_n:n \geq 1\}$.
Inf-liminf Theorem \ref{slst} can be applied here to obtain that
$$
\inf_{\overline{A}^{\sigma(E^{**}, E^*)}}\liminf_n y^*_n = \inf_{A}\liminf_n y^*_n
$$
which contradicts \eqref{c1} and \eqref{c2} since both inequalities are valid for $(y^*_n)$ instead of $(x^*_n)$. Indeed
we have:
$$\inf_{A}\liminf_n y^*_n=\inf_{A}\lim_n y^*_n=\inf_A x_0^{*}\geq \beta$$
and
$$\inf_{\overline{A}^{\sigma(E^{**}, E^*)}}\liminf_n y^*_n\leq \liminf_n y^*_n(x^{**}_0)\leq \alpha < \beta$$
\end{proof}

\begin{rem}
Let us remark here that Theorem A.1 in \cite{jou-sch-tou} is a consequence of the former result in separable Banach spaces.
Indeed, if we set $A$ equal to the epigraph of the penalty function  $V$ there, it is enough to consider the singleton set $D=\{0,1\}$
to get the result. 
\end{rem}

\section{Conic James' construction} 

We present in this section the main construction for our proof of Theorem \ref{MainConic}.
A carefull analysis of  ideas of R.C. James \cite{jam} lead us to  find an appropiate (conic-sided) non attaining linear form.

\begin{thm}\label{ori}
Let $A$ be a convex bounded subset of a Banach space $E$ such
that the set
$$
\{a^{**}\in \overline{A}^{w^*}\backslash E : \hbox{there exists }
\{a_n\}_{n\ge 1} \hbox{ in } A \hbox{ with } a^{**}\in
\overline{\{a_n:n\ge 1\}}^{w^*} \}
$$
is non--void. Let us fix a convex weakly compact subset $D$ of $E$ 
which doest not contains the origin. 

Then there is a sequence
$\left\{x_n^*\right\}_{n\ge 1}$ in $B_{E^*}$ and $g_0^*\in
\hbox{co}_{\sigma}\{x_n^*:n\ge 1\}$ such that for all $h\in
\ell_\infty (A)$ satisfying that for all $a\in A,$
$$
\liminf_{n\ge 1}x_n^*(a) \le h(a) \le\limsup_{n\ge 1}x_n^*(a),
$$
we have that
$$
g_0^*- h\hbox{ does not attain its supremum on } A.
$$
and $( g_0^*- h)(d)>0$ for every $d\in D$
\end{thm}

\begin{proof}

By hypothesis we can  fix a point $x_0^{**}\in\overline{A}^{w^*}\backslash E$ 
and $\{x_n\}_{n\ge 1}$ a
sequence in $A$ with $x_0^{**}\in \overline{\{x_n:n\ge
1\}}^{w^*}$.
The line segment $[-x^{**}_0, x^{**}_0]$ doest not meet the $\sigma(E, E^*)$-compact set $D$ since the origin is not inside $D$. Let us take 
$x^*\in B_{E^{*}}$ which strictly separates both sets, we will have:
$$x^*([-x^{**}_0, x^{**}_0])=[-\mu, +\mu]$$
and 
$$x^*(D)= [a,b]$$ where $\mu , a, b$ are real numbers and $[-\mu, \mu]\cap [a,b]=\emptyset.$
Without  loss of generality we assume that
$$-\mu\leq 0\leq \mu < \beta <\alpha < a < b$$
for $\alpha$ and $\beta$ real numbers choosed by the strict separation provided with the linear form $x^*$.

Without loss of generality we may and do assume the fact that 
\begin{equation}\label{08}
x^*(x_n)< \beta \hbox{ for every } n\in \N.
\end{equation}

Another Hahn--Banach application, now in the duality $(E^{**}, E^{***})$, provides us a linear form $z^{***}\in B_{E^{***}}$ such that
$z^{***}(x^{**}_0)>0$ but $z^{***}(E)=\{0\}$. Let us choose $\lambda >0$ such that 
$\lambda z^{***}(x^{**}_0)+ x^*(x^{**}_0 )> \alpha$

Let us consider the linear functional $$x^{***}:= x^*+ \lambda z^{***}$$ and we will have 
\begin{equation} \label{01}
x^{***}(D)=x^*(D)\geq a > \alpha >0,
x^{***}(x^{**}_0) = \lambda z^{***}(x^{**}_0)+ x^*(x^{**}_0 )> \alpha >0
\end{equation}

Let us remind the reader that we also have by (\ref{08})
\begin{equation} \label{02}
x^{***}(x_n)=x^{*}(x_n) <\beta 
\end{equation}  
for every $n\in \N$.
Without loss of generality, taking $\frac{1}{\|x^{***}\|}x^{***}, \frac{1}{\|x^{***}\|}\alpha$ and $\frac{1}{\|x^{***}\|}\beta$ instead of $x^{***}, \alpha$ and $\beta$ if necessary, we may assume the former inequalities 
(\ref{01}) and (\ref{02}) with
$x^{***} \in B_{E^{***}}$.

By Goldstine and Mackey-Arens theorems we can construct a bounded
sequence $\left\{x_n^*\right\}_{n\ge 1}$ in $B_{E^*}$ satisfying
$$
\hbox{for all } p\ge 1, \qquad \lim_{n\ge
1}x_n^*(x_p)=x^{***}(x_p) = x^*(x_p)  ,
$$
and
$$
\lim_{n\ge 1}x_0^{**}(x_n^*)=x^{***}(x_0^{**}),
$$
$$
\lim_{n\ge 1}(x_n^*(d))=x^{***}(d) = x^*(d)
$$
uniformly on $d\in D$.
Then
\begin{equation}\label{03}
\hbox{for all } p\ge 1, \hbox{ there is } n_p \hbox{ such that } 
\beta > x_n^*(x_p) \hbox{ for } n\geq n_p ,
\end{equation}

and without loss of generality we may assume that:
\begin{equation}\label{04}
x_0^{**}(x_n^*)> \alpha, 
x^*_n(d) > \alpha , \forall n \in \N \hbox{ and }  d \in D
\end{equation}
by 
uniform convergence on the weak compact set  $D$.

Let us note that, given a $\sigma(E^*,E)$--cluster point $x_0^*$
of the sequence $\left\{x_n^*\right\}_{n\ge 1}$, we have that
\begin{equation}\label{05}
x_0^{**}(x_0^*)\leq \beta,
\end{equation}

because $x_0^{**}\in \overline{\{x_p:p\ge 1\}}^{w^*}$ and for all
$p\ge 1, x_0^*(x_p) \leq \beta$ by (\ref{03}).

It is time now to Pryce
arguments (see \cite{gal-sim} Lemma 9,c; Proposition 10.14 and Theorem 10.15 in \cite{cas-ori-rui}):
there is a subsequence
$\left\{x^*_{n_k}:k\ge 1\right\}$ such that

$\hbox{for all } h_0\in \hbox{co}_{\sigma}\left\{ x^*_{n}: n\ge 1
\right\},$

we have:
\begin{equation}\label{06}
\sup_A\left( h_0-\limsup_{k\ge
1}x^*_{n_k}\right)= \sup_A \left(h_0-\liminf_{k\ge
1}x^*_{n_k}\right).
\end{equation}
Let us now observe that for $x_0^{**}$ we
have by (\ref{04}) that
\begin{equation}\label{07}
\forall  h_0\in \hbox{co}_{\sigma}\left\{ x^*_{n_k}: k\ge 1
\right\}, \qquad x_0^{**}(h_0) >\alpha.
\end{equation}

Let us fix a $\sigma(E^*,E)$-cluster point $x^*_0$  of  the sequence $\left\{
x^*_{n_k}:k\ge 1 \right\}$; then it follows that for
all $a \in A$,
$$
\limsup_{k\ge 1}x^*_{n_k}(a) \ge x_0^*(a) \ge \liminf_{k\ge
1}x^*_{n_k}(a)
$$
and thus, for all $a \in A$,
$$
\left( h_0(a)-\liminf_{k\ge 1}x^*_{n_k}(a) \right)\ge $$ 
$$\left(h_0-x^*_0)(a)\right) \ge \left(h_0(a)- \limsup_{k\ge 1}x^*_{n_k}(a)\right).
$$
Therefore, in view of (\ref{06}) we deduce that
$
\hbox{for all } h_0, $

$$\sup_A \left(h_0-\limsup_{k\ge
1}x^*_{n_k}\right)= 
\sup_A \left(h_0-\liminf_{k\ge
1}x^*_{n_k}\right)= \sup_A \left(h_0-x_0^*\right)
.$$
Let us observe that for $h_0\in \hbox{co}_\sigma \left\{ x^*_{n_k}:k\ge 1 \right\}$ we
have:
$$
\sup_A \left(h_0-x_0^*\right)=\sup_{\overline{A}^{w^*}}
\left(h_0-x_0^*\right)\ge \left(x_0^{**}(h_0)-x_0^{**}(x_0^{*})\right)= \langle x_0^{**}, h_0 - x_0^*  \rangle$$

and  

$$\langle x_0^{**}, h_0 - x_0^*  \rangle  > \alpha -\beta >0 , $$
by (\ref{05}) and (\ref{07}), as needed to apply  [5, Corollary 8] and obtain a sequence
$\left\{g_i^*\right\}_{i\ge 1}$ with $g_i \in \hbox{co}_\sigma \{x_{n_k}^* : k\ge
i \}$ and  $g_0^*\in \hbox{co}_{\sigma}\left\{ g_i^*:i\ge 1 \right\}$
such that for all $\tilde{g}\in \ell_\infty (A)$ with
$$\liminf_{i\ge 1} g_i^* \le \tilde{g} \le
\limsup_{i\ge 1}g_i^*\ \hbox{on } A
$$
we have that
$$
(g_0^*-\tilde{g}) \hbox{ does not attain its supremum on } A.
$$

Let us observe that $\tilde{g}$ does coincide with $x^{***}$ on $D$ and that $g_i$ are $\sigma$-convex combinations of linear forms of the sequence
$\{x^*_n\}$. If we may assume in the construction that
$x^*_n(d)> x^{***}(d)$ for all $d\in D$ and every $n \in \N$, then we should have $g_i(d) > x^{***}(d)$ for all $d\in D$, therefore $(g_0^*-\tilde{g})(d)>0$ for every $d\in D$ and  the proof should be over. In particular, for every $\sigma(E^*,E)$--cluster point of the sequence
$\left\{g_i^* \right\}_{i\ge 1}$, let us say $\tilde{g}^*,$ we have
that $g_0^*-\tilde{g}^*\in E^*$ does not attain its supremum on $A$ but $(g_0^*-\tilde{g}*)(d)>0$ for every $d\in D$ .

Let us finish with the proof of our claim:

{\it It is possible to choose $x^*_n$ above such that $$x^*_n(d)> x^{***}(d)$$ for every $d\in \N$.}

We set $$\epsilon_n:= 2
\sup\{|x^{***}(d)- x^*_n (d)|: d\in D \}$$and $$c_n:= \sup\{\frac{x^{***}(d)}{x^{***}(d)-\epsilon_n} : d\in D\}$$
that is well defined because without loss of generality we may assume that $$\epsilon_n<a=\inf x^{***}(D)$$ for every $n$.

The function $f(t,\epsilon):= \frac{t}{t-\epsilon}$ is decreasing in $t>\epsilon$ for $\epsilon >0$ fixed, and 
$\lim_{\epsilon\rightarrow 0} f(t,\epsilon) =1$ for every $t>0$ fixed too. If we set $\hat{x}^*_n:= c_n x^*_n$, we will have that
the sequence $\hat{x}^*_n$ verifies our claim since $\epsilon_n$ goes to zero and $c_n$ to $1$ when $n$ goes to infinity. Indeed the proof is over.

\end{proof}

\bigskip

Now  we are in conditions to give the proof of our Theorem \ref{MainConic}
\begin{proof}
If $\hbox{co}(A)$ is not relatively weakly compact it satisfies the hypothesis of Theorem \ref{ori}, thus we find $x^*\in E^*$ such that $x^*$ doest not attains its supremum on  $\overline{\co{(A)}}^{\|\cdot\|}$ but $x^*(D)>0$, which is a contradiction with our hypothesis and it finishes the proof.
\end{proof}

Theorem \ref{MainConic} extends Theorem 10 of \cite{cas-ori-per} for arbitrary Bananch spaces solving a question asked in that paper.
Unfortunately Theorem 2 of \cite{cas-ori-per} remains unknown for arbitrary Banach spaces.

\end{document}